\documentclass[reqno,11pt]{article} 
\usepackage[margin=1in]{geometry}  
\usepackage{amsmath, amssymb}
\usepackage{amsthm} 
\usepackage{color}
\newcommand{\Mod}[1]{\ (\textup{mod}\ #1)}
\theoremstyle{plain} 
\newtheorem{theorem}{\indent\sc Theorem}[section]
\newtheorem{lemma}[theorem]{\indent\sc Lemma}

\newtheorem{proposition}[theorem]{\indent\sc Proposition}

\theoremstyle{definition} 

\newtheorem{remark}[theorem]{\indent\sc Remark}

\makeatletter
\def\address#1#2{\begingroup
\noindent\parbox[t]{7.8cm}{%
\small{\scshape\ignorespaces#1}\par\vskip1ex
\noindent\small{\itshape E-mail address}%
\/: #2\par\vskip4ex}\hfill%
\endgroup}%
\makeatother
\title{On a problem of Hasse and Ramachandra} 
\author{
\textsc{Ja Kyung Koo, Dong Hwa Shin and Dong Sung Yoon$^*$} 
}
\date{} 
\begin{document}

\allowdisplaybreaks

\maketitle

 \footnote{
\thanks{
$^*$Corresponding author.\\
The second author was supported by the National Research Foundation of Korea (NRF)
grant funded by the Korea government (MSIP) (2017R1A2B1006578), and by Hankuk
University of Foreign Studies Research Fund of 2017. 
The third (corresponding) author was
supported by Basic Science Research Program through the National Research Foundation of
Korea(NRF) funded by the Ministry of Education (2017R1D1A1B03030015).
}}

\begin{abstract}
Let $K$ be an imaginary quadratic field, and let $\mathfrak{f}$ be
a nontrivial integral ideal of $K$.
Hasse and Ramachandra asked whether
the ray class field of $K$ modulo $\mathfrak{f}$ can be
generated by a single value of the Weber function.
We completely resolve this question when $\mathfrak{f}=(N)$
for an integer $N>1$.
\end{abstract}

\maketitle

\section {Introduction}

Let $K$ be an imaginary quadratic field with ring of integers $\mathcal{O}_K$, and let
$E$ be an elliptic curve
with complex multiplication by $\mathcal{O}_K$.
When $E$ is given by the affine model
\begin{equation*}
y^2=4x^3-g_2x-g_3\quad\textrm{with}~
g_2=g_2(\mathcal{O}_K)~\textrm{and}~g_3=g_3(\mathcal{O}_K),
\end{equation*}
the \textit{Weber function}
$h:\mathbb{C}/\mathcal{O}_K\rightarrow\mathbb{P}^1(\mathbb{C})$ is defined by
\begin{equation}\label{Weberdef}
h(z)=\left\{\begin{array}{ll}
(g_2^2/\Delta)\wp(z)^2 & \textrm{if}~K=\mathbb{Q}(\sqrt{-1}),\\
(g_3/\Delta)\wp(z)^3 & \textrm{if}~K=\mathbb{Q}(\sqrt{-3}),\\
(g_2g_3/\Delta)\wp(z) & \textrm{otherwise},
\end{array}\right.
\end{equation}
where $\Delta=g_2^3-27g_3^2$ and $\wp(z)=\wp(z;\,\mathcal{O}_K)$.
This map gives rise to an isomorphism
of $E/\mathrm{Aut}(E)$ onto $\mathbb{P}^1(\mathbb{C})$
(\cite[Theorem 7 in Chapter 1]{Lang87}).
\par
Let $\mathfrak{f}$ be a proper nontrivial ideal of $\mathcal{O}_K$.
We denote by $H$ the Hilbert class field of $K$, and by
$K_\mathfrak{f}$ the ray class field of $K$ modulo $\mathfrak{f}$.
As a consequence of the main theorem of the theory of complex multiplication,
Hasse proved in \cite{Hasse} that
\begin{equation}\label{CM}
H=K(j)~\textrm{with}~
j=1728\,\frac{g_2^3}{\Delta}\quad\textrm{and}\quad
K_\mathfrak{f}=H\left(h(z_0)\right)~\textrm{for some}~
z_0\in\mathfrak{f}^{-1}.
\end{equation}
See also \cite[Chapter 10]{Lang87}.
In his letter to Hecke, Hasse further asked
whether $K_\mathfrak{f}$ can be generated by a single value of $h$ without
the $j$-invariant (\cite[p. 91]{F-L-R}),
and Ramachandra also mentioned this problem later in \cite{Ramachandra}.
It was Sugawara who first gave a partial answer to this question
(\cite{Sugawara33} and \cite{Sugawara36}),
however, it still remains an open question.
\par
Now, let $\mathfrak{f}=(N)$ for an integer $N>1$. Note by (\ref{CM}) that 
the question is trivial if the class number of $K$ is one.
Recently, the case $\gcd(N,\,6)=1$ was dealt with in \cite{K-S-Y}
by the same authors of this paper, and the case $N\in\{2,\,3,\,4,\,6\}$
will be appeared in \cite{J-K-S}.
In this paper by modifying the idea of \cite{K-S-Y}, that is, through careful understanding
about the characters on class groups and the  second Kronecker limit formula 
we shall eventually resolve
Hasse-Ramachandra's problem for $\mathfrak{f}=(N)$.

\section {The second Kronecker limit formula}

For $\mathbf{v}=\begin{bmatrix}r_1\\r_2\end{bmatrix}\in(\mathbb{Q}\setminus\mathbb
{Z})^2$, we define the (first) \textit{Fricke function} $f_{\mathbf{v}}(\tau)$ on the upper half-plane 
$\mathbb{H}$ by
\begin{equation}\label{Frickedef}
f_{\mathbf{v}}(\tau)=
\frac{g_2(\tau)g_3(\tau)}{\Delta(\tau)}\,\wp(r_1\tau+r_2),
\end{equation}
where $g_2(\tau)=g_2([\tau,\,1])$,
$g_3(\tau)=g_3([\tau,\,1])$, $\Delta(\tau)=
\Delta([\tau,\,1])$ and $\wp(z)=\wp(z;\,[\tau,\,1])$.
This function depends only on $\pm\mathbf{v}\Mod{\mathbb{Z}^2}$, 
and is holomorphic on $\mathbb{H}$ (\cite[Chapters 3 and 6]{Lang87}).
Furthermore, we define the \textit{Siegel function}
$g_{\mathbf{v}}(\tau)$ on $\mathbb{H}$
by the following infinite product
\begin{equation*}
g_{\mathbf{v}}(\tau)
=-e^{\pi\mathrm{i}r_2(r_1-1)}q^{(1/2)(r_1^2-r_1+1/6)}
(1-q^{r_1}e^{2\pi\mathrm{i}r_2})\prod_{n=1}^\infty (1-q^{n+r_1}e^{2\pi\mathrm{i}r_2})
(1-q^{n-r_1}e^{-2\pi\mathrm{i}r_2}),
\end{equation*}
where $q=e^{2\pi\mathrm{i}\tau}$.
If $N$ is a positive integer so that $N\mathbf{v}\in\mathbb{Z}^2$, then 
$g_{\mathbf{v}}(\tau)^{12N}$
depends only on $\pm\mathbf{v}\Mod{\mathbb{Z}^2}$, 
and has neither zeros nor poles on $\mathbb{H}$ (\cite[$\S$2.1]{K-L}).

\begin{lemma}\label{FrickeSiegel}
Let $\mathbf{u},\,\mathbf{v}\in(\mathbb{Q}\setminus\mathbb{Z})^2$ such that
$\mathbf{u}\not\equiv\pm\mathbf{v}\Mod{\mathbb{Z}^2}$. Then we have the relation
\begin{equation*}
\left(f_\mathbf{u}(\tau)-
f_\mathbf{v}(\tau)\right)^6=\frac{j(\tau)^2(j(\tau)-1728)^3}{2^{30}3^{24}}
\frac{g_{\mathbf{u}+\mathbf{v}}(\tau)^6g_{\mathbf{u}-\mathbf{v}}(\tau)^6}
{g_\mathbf{u}(\tau)^{12}g_\mathbf{v}(\tau)^{12}}.
\end{equation*}
\end{lemma}
\begin{proof}
See \cite[Theorem 2 in Chapter 18]{Lang87} and \cite[p. 29 and p. 51]{K-L}.
\end{proof}

Let $K$ be an imaginary quadratic field, let
$\mathfrak{f}$ be a proper nontrivial ideal of $\mathcal{O}_K$ and let $N$ ($>1$) be the smallest positive integer in $\mathfrak{f}$.
We denote by $\mathrm{Cl}(\mathfrak{f})$ the ray class group of $K$ modulo $\mathfrak{f}$.
Then $\mathrm{Gal}(K_\mathfrak{f}/K)$ is isomorphic to $\mathrm{Cl}(\mathfrak{f})$ via the Artin map $\sigma=\sigma_\mathfrak{f}:\mathrm{Cl}(\mathfrak{f})
\rightarrow\mathrm{Gal}(K_\mathfrak{f}/K)$.
Let $C\in\mathrm{Cl}(\mathfrak{f})$.
Take any integral ideal $\mathfrak{c}$ in the class $C$ and express
\begin{align*}
\mathfrak{f}\mathfrak{c}^{-1}&=[\omega_1,\,\omega_2]\quad\textrm{for some}~ \omega_1,\,\omega_2\in\mathbb{C}~\textrm{such that}~\omega=\frac{\omega_1}{\omega_2}\in\mathbb{H},\\
1&=r_1\omega_1+r_2\omega_2\quad\textrm{for some}~r_1,\,r_2\in(1/N)\mathbb{Z}.
\end{align*}
We define the
\textit{Fricke invariant} $f_\mathfrak{f}(C)$ and the
\textit{Siegel-Ramachandra invariant} $g_\mathfrak{f}(C)$ by
\begin{equation}\label{S-Rinvariant}
f_\mathfrak{f}(C)=f_{\left[\begin{smallmatrix}r_1\\r_2\end{smallmatrix}\right]}(\omega)
\quad\textrm{and}\quad
g_\mathfrak{f}(C)=g_{\left[\begin{smallmatrix}r_1\\r_2\end{smallmatrix}\right]}(\omega)^{12N},
\end{equation}
respectively.
These values depend only on the class $C$, not on the choices of $\mathfrak{c}$, $\omega_1$ and $\omega_2$ (\cite[$\S$6.2 and $\S$6.3]{Lang87} and \cite[$\S$2.1 and 11.1]{K-L}).

\begin{proposition}\label{transformation}
The  invariants $f_\mathfrak{f}(C)$ and $g_\mathfrak{f}(C)$ belong to $K_\mathfrak{f}$.
Furthermore, they satisfy
\begin{equation*}
f_\mathfrak{f}(C)^{\sigma(C')}=f_\mathfrak{f}(CC')
\quad\textrm{and}\quad
g_\mathfrak{f}(C)^{\sigma(C')}=g_\mathfrak{f}(CC')\quad\textrm{for all}~C'\in\mathrm{Cl}(\mathfrak{f}).
\end{equation*}
\end{proposition}
\begin{proof}
See \cite[Theorem 1.1 in Chapter 11]{K-L}.
\end{proof}

Let $\chi$ be a nonprincipal character of $\mathrm{Cl}(\mathfrak{f})$.
We define the \textit{Stickelberger
element} $S(\chi)=S_\mathfrak{f}(\chi)$ by
\begin{equation}\label{Stickelberger}
S(\chi)=
\sum_{C\in\mathrm{Cl}(\mathfrak{f})}
\chi(C)\ln|g_\mathfrak{f}(C)|,
\end{equation}
and  the \textit{$L$-function} $L_\mathfrak{f}(s,\,\chi)$ by
\begin{equation*}
L_\mathfrak{f}(s,\,\chi)=
\sum_{\mathfrak{a}}\frac{\chi([\mathfrak{a}])}{\mathrm{N}_{K/\mathbb{Q}}(\mathfrak{a})^s}\quad(s\in\mathbb{C}),
\end{equation*}
where $\mathfrak{a}$ runs over all nontrivial ideals of $\mathcal{O}_K$ prime to $\mathfrak{f}$
and $[\mathfrak{a}]$ stands for the class in $\mathrm{Cl}(\mathfrak{f})$ containing
the ideal $\mathfrak{a}$.
We shall denote by $\mathfrak{f}_\chi$ the conductor of the character $\chi$.

\begin{proposition}\label{Kronecker}
Let $\chi_0$ be the primitive character of $\chi$ on $\mathrm{Cl}(\mathfrak{f}_\chi)$.
If $\mathfrak{f}_\chi\neq\mathcal{O}_K$, then we obtain the relation
\begin{equation*}
\left(\prod_{\begin{smallmatrix}\mathfrak{p}~:~\textrm{prime ideals of $\mathcal{O}_K$}\\
~~~\textrm{such that}~\mathfrak{p}\,|\,\mathfrak{f},~\mathfrak{p}\,\nmid\,\mathfrak{f}_\chi\end{smallmatrix}}
(1-\overline{\chi_0}([\mathfrak{p}]))\right)
L_{\mathfrak{f}_\chi}(1,\,\chi_0)=-
\frac{\pi\chi_0([\gamma\mathfrak{d}_K\mathfrak{f}_\chi])}{3N(\mathfrak{f}_\chi)\sqrt{|d_K|}\omega(\mathfrak{f}_\chi)
T_\gamma(\overline{\chi_0})}S(\overline{\chi}),
\end{equation*}
where $\mathfrak{d}_K$ is
the different ideal of the extension $K/\mathbb{Q}$, $\gamma$ is an element of $K$ so that
$\gamma\mathfrak{d}_K\mathfrak{f}_\chi$ is a nontrivial ideal of $\mathcal{O}_K$ prime to $\mathfrak{f}_\chi$,
$N(\mathfrak{f}_\chi)$ is the least positive integer in $\mathfrak{f}_\chi$,
$\omega(\mathfrak{f}_\chi)
=|\{\alpha\in\mathcal{O}_K^*~|~\alpha\equiv1\Mod{\mathfrak{f}_\chi}\}|$ and
\begin{eqnarray*}
T_\gamma(\overline{\chi_0})=
\sum_{\alpha+\mathfrak{f}_\chi\in(\mathcal{O}_K/\mathfrak{f}_\chi)^*}
\overline{\chi_0}([\alpha\mathcal{O}_K])e^{2\pi
\mathrm{i}\mathrm{Tr}_{K/\mathbb{Q}}(\alpha\gamma)}.
\end{eqnarray*}
\end{proposition}
\begin{proof}
See \cite[Theorem 9 in Chapter II]{Siegel} or \cite[Theorem
2.1 in Chapter 11]{K-L}.
\end{proof}

\begin{remark}\label{Eulerfactor}
Since $\chi_0$ is a nonprincipal character of $\mathrm{Cl}(\mathfrak{f}_\chi)$
by the assumption $\mathfrak{f}_\chi\neq\mathcal{O}_K$, we have $L_{\mathfrak{f}_\chi}(1,\,\chi_0)\neq0$
(\cite[Theorem 10.2 in Chapter V]{Janusz}).
Thus, if every prime ideal factor of $\mathfrak{f}$ divides $\mathfrak{f}_\chi$,
then we derive by Proposition \ref{Kronecker} that $S(\overline{\chi})\neq0$.
\end{remark}

\section {Differences of Weber functions}

For an imaginary quadratic field $K$, 
fix an element $\tau_K$ of $\mathbb{H}$ so that
$\mathcal{O}_K=[\tau_K,\,1]$.
From now on, we assume that $K$ is different from $\mathbb{Q}(\sqrt{-1})$ and $\mathbb{Q}(\sqrt{-3})$, and let $N>1$. We then have $j(\tau_K)\neq0,\,1728$ (\cite[p. 261]{Cox}) and
\begin{equation*}
h(r_1\tau_K+r_2)=f_{\left[\begin{smallmatrix}r_1\\r_2\end{smallmatrix}\right]}(\tau_K)
\quad\textrm{for all}~\begin{bmatrix}r_1\\r_2\end{bmatrix}\in(\mathbb{Q}\setminus\mathbb{Z})^2
\end{equation*}
by the definitions (\ref{Weberdef}) and (\ref{Frickedef}).
\par
Let $H_N$
be the ring class field of the 
order of conductor $N$ in $K$.
Then we have a tower of fields
\begin{equation*}
K\subseteq H\subseteq H_N\subseteq K_{(N)}
\end{equation*}
(\cite[$\S$7]{Cox}).
For an integer $t$ prime to $N$, by $C_t=C_{N,\,t}$ we mean the class in the ray class group $\mathrm{Cl}(N)$ of $K$ modulo $(N)$ containing the ideal $(t)$. Note that $C_1$ is the identity element of $\mathrm{Cl}(N)$. 

\begin{lemma}\label{tinvariant}
If $t$ is an integer prime to $N$, then we get
\begin{equation*}
f_{(N)}(C_t) =f_{\left[\begin{smallmatrix}0\\t/N\end{smallmatrix}\right]}(\tau_K)
\quad\textrm{and}\quad
g_{(N)}(C_t) =g_{\left[\begin{smallmatrix}0\\t/N\end{smallmatrix}\right]}(\tau_K)^{12N}.
\end{equation*}
\end{lemma}
\begin{proof}
Since
\begin{equation*}
(N\mathcal{O}_K)(t\mathcal{O}_K)^{-1}=(N/t)\mathcal{O}_K=
[N\tau_K/t,\,N/t]
\quad\textrm{and}\quad
1=0(N\tau_K/t)+(t/N)(N/t),
\end{equation*}
we deduce the lemma by the definition (\ref{S-Rinvariant}).
\end{proof}

For an intermediate field $F$ of the extension $K_{(N)}/K$, we shall denote by $\mathrm{Cl}(K_{(N)}/F)$ the subgroup of $\mathrm{Cl}(N)$
corresponding to $\mathrm{Gal}(K_{(N)}/F)$.

\begin{lemma}\label{ringGalois}
We have
\begin{equation*}
\mathrm{Cl}(K_{(N)}/H_N)=
\{C_t~|~t\in(\mathbb{Z}/N\mathbb{Z})^*/\{\pm 1\}\}
\simeq(\mathbb{Z}/N\mathbb{Z})^*/\{\pm 1\}.
\end{equation*}
\end{lemma}
\begin{proof}
See \cite[Proposition 3.8]{E-K-S}.
\end{proof}

Let $t$ be an integer such that
\begin{equation*}
\gcd(N,\,t)=1\quad\textrm{and}\quad t\not\equiv\pm1\Mod{N}.
\end{equation*}
Note that such an integer $t$ always exists except for the four cases $N=2,\,3,\,4,\,6$. Express $(t+1)/N$ and $(t-1)/N$ as
\begin{equation*}
\frac{t+1}{N}=\frac{n_+}{N_+}\quad\textrm{and}
\quad\frac{t-1}{N}=\frac{n_-}{N_-}\,,
\end{equation*}
where $n_+,\,N_+,\,n_-,\,N_-$ are integers such that
$N_+,\,N_->0$ and $\gcd(n_+,N_+)=\gcd(n_-,N_-)=1$.
Observe that the condition $t\not\equiv\pm1\Mod{N}$ is equivalent to
saying that neither $N_+$ nor $N_-$ is equal to $1$.
\par

Now, we define
\begin{equation}\label{xit}
\xi_t=\left(h(t/N)-h(1/N)\right)^{12N}=\left(f_{\left[\begin{smallmatrix}0\\t/N\end{smallmatrix}\right]}(\tau_K)
-f_{\left[\begin{smallmatrix}0\\1/N\end{smallmatrix}\right]}(\tau_K)\right)^{12N}.
\end{equation}
Furthermore, for a character $\chi$ of $\mathrm{Cl}(N)$ we denote by
\begin{equation*}
S(\chi,\,\xi_t)=
\sum_{C\in\mathrm{Cl}(N)}\chi(C)\ln\left|\xi_t^{\sigma(C)}\right|.
\end{equation*}

\begin{lemma}\label{summation}
If $\chi$ is nontrivial on $\mathrm{Cl}(K_{(N)}/H)$, then we obtain
\begin{eqnarray*}
S(\overline{\chi},\,\xi_t)&=&
(N/N_+)\sum_{\substack{B_+\in\mathrm{Cl}(N)\\\Mod{\mathrm{Cl}(K_{(N)}/K_{(N_+)})}}}
\overline{\chi}(B_+)\ln\left|g_{(N_+)}(C_{N_+,\,n_+})^{\sigma(B_+)}\right|
\sum_{A_+\in\mathrm{Cl}(K_{(N)}/K_{(N_+)})}\overline{\chi}(A_+)\\
&&+(N/N_-)\sum_{\substack{B_-\in\mathrm{Cl}(N)\\\Mod{\mathrm{Cl}(K_{(N)}/K_{(N_-)})}}}
\overline{\chi}(B_-)\ln\left|g_{(N_-)}(C_{N_-,\,n_-})^{\sigma(B_-)}\right|
\sum_{A_-\in\mathrm{Cl}(K_{(N)}/K_{(N_-)})}\overline{\chi}(A_-)\\
&&-2(\chi(C_t)+1)S(\overline{\chi}).
\end{eqnarray*}
\end{lemma}
\begin{proof}
We derive that
\begin{eqnarray*}
S(\overline{\chi},\,\xi_t)
&=&
\sum_{C\in\mathrm{Cl}(N)}
\overline{\chi}(C)
\ln\left|\left(
\frac{j(\tau_K)^{4N}(j(\tau_K)-1728)^{6N}}{2^{60N}3^{48N}}
\right)^{\sigma(C)}
\right|\\
&&+\sum_{C\in\mathrm{Cl}(N)}
\overline{\chi}(C)\ln\left|\left(g_{\left[\begin{smallmatrix}0\\n_+/N_+\end{smallmatrix}\right]}(\tau_K)^{12N}
\right)^{\sigma(C)}\right|
+\sum_{C\in\mathrm{Cl}(N)}
\overline{\chi}(C)\ln\left|\left(g_{\left[\begin{smallmatrix}0\\n_-/N_-\end{smallmatrix}\right]}(\tau_K)^{12N}
\right)^{\sigma(C)}\right|
\\
&&
-\sum_{C\in\mathrm{Cl}(N)}
\overline{\chi}(C)\ln\left|\left(g_{\left[\begin{smallmatrix}0\\t/N\end{smallmatrix}\right]}(\tau_K)^{24N}
\right)^{\sigma(C)}\right|
-\sum_{C\in\mathrm{Cl}(N)}
\overline{\chi}(C)\ln\left|\left(g_{\left[\begin{smallmatrix}0\\1/N\end{smallmatrix}\right]}(\tau_K)^{24N}
\right)^{\sigma(C)}\right|
\\
&&\textrm{by the definition (\ref{xit}) and Lemma \ref{FrickeSiegel}}\\
&=&
\sum_{\substack{B\in\mathrm{Cl}(N)\\\Mod{\mathrm{Cl}(K_{(N)}/H)}}}
\sum_{A\in\mathrm{Cl}(K_{(N)}/H)}\overline{\chi}(AB)
\ln\left|\left(
\frac{j(\tau_K)^{4N}(j(\tau_K)-1728)^{6N}}{2^{60N}3^{48N}}
\right)^{\sigma(AB)}
\right|\\
&&+
(N/N_+)\sum_{\substack{B_+\in\mathrm{Cl}(N)\\\Mod{\mathrm{Cl}(K_{(N)}/K_{(N_+)})}}}
\sum_{A_+\in\mathrm{Cl}(K_{(N)}/K_{(N_+)})}\overline{\chi}(A_+B_+)
\ln\left|g_{(N_+)}(C_{N_+,\,n_+})^{\sigma(A_+B_+)}
\right|\\
&&+(N/N_-)
\sum_{\substack{B_-\in\mathrm{Cl}(N)\\\Mod{\mathrm{Cl}(K_{(N)}/K_{(N_-)})}}}
\sum_{A_-\in\mathrm{Cl}(K_{(N)}/K_{(N_-)})}\overline{\chi}(A_-B_-)
\ln\left|g_{(N_-)}(C_{N_-,\,n_-})^{\sigma(A_-B_-)}
\right|\\
&&
-2\sum_{C\in\mathrm{Cl}(N)}
\overline{\chi}(C)\ln\left|g_{(N)}(C_t)
^{\sigma(C)}\right|
-2\sum_{C\in\mathrm{Cl}(N)}
\overline{\chi}(C)\ln\left|g_{(N)}(C_1)^{\sigma(C)}\right|
\quad\textrm{by Lemma \ref{tinvariant}}\\
&=&
\sum_{B}
\overline{\chi}(B)
\ln\left|\left(
\frac{j(\tau_K)^{4N}(j(\tau_K)-1728)^{6N}}{2^{60N}3^{48N}}
\right)^{\sigma(B)}
\right|
\sum_{A}\overline{\chi}(A)\\
&&+(N/N_+)
\sum_{B_+}
\overline{\chi}(B_+)\ln\left|g_{(N_+)}(C_{N_+,\,n_+})^{\sigma(B_+)}\right|
\sum_{A_+}\overline{\chi}(A_+)\\
&&+(N/N_-)\sum_{B_-}
\overline{\chi}(B_-)\ln\left|g_{(N_-)}(C_{N_-,\,n_-})^{\sigma(B_-)}\right|
\sum_{A_-}\overline{\chi}(A_-)\\
&&
-2\chi(C_t)\sum_{C}
\overline{\chi}(C_tC)\ln\left|g_{(N)}(C_tC)\right|
-2\sum_{C}
\overline{\chi}(C)\ln\left|g_{(N)}(C)\right|\quad\textrm{by (\ref{CM}) and Proposition \ref{transformation}}\\
&=&
(N/N_+)\sum_{B_+}
\overline{\chi}(B_+)\ln\left|g_{(N_+)}(C_{N_+,\,n_+})^{\sigma(B_+)}\right|
\sum_{A_+}\overline{\chi}(A_+)\\
&&+(N/N_-)\sum_{B_-}
\overline{\chi}(B_-)\ln\left|g_{(N_-)}(C_{N_-,\,n_-})^{\sigma(B_-)}\right|
\sum_{A_-}\overline{\chi}(A_-)\\
&&-2(\chi(C_t)+1)S(\overline{\chi})\\
&&\textrm{by the assumption that $\chi$ is nontrivial on $\mathrm{Cl}(K_{(N)}/H)$
and the definition (\ref{Stickelberger})}.
\end{eqnarray*}
\end{proof}

\section {Lemmas on characters of class groups}

If we set
\begin{equation*}
F=K\left(h(1/N)\right)=K\left(
f_{\left[\begin{smallmatrix}0\\1/N\end{smallmatrix}\right]}(\tau_K)\right),
\end{equation*}
then we obtain by (\ref{CM}) that
\begin{equation}\label{intersection}
\mathrm{Cl}(K_{(N)}/H)\cap
\mathrm{Cl}(K_{(N)}/F)
=\mathrm{Cl}(K_{(N)}/HF)
=\mathrm{Cl}(K_{(N)}/K_{(N)})=\{C_1\}.
\end{equation}
In this section, we shall prove the existence of certain characters of class groups
under the assumption that $F$ is properly contained in $K_{(N)}$.

\begin{lemma}\label{character}
Assume that
\begin{equation*}
\gcd(72,\,N)\in\{1,\,8,\,9,\,72\}.
\end{equation*}
Then, there is a character $\chi$ of $\mathrm{Cl}(N)$ satisfying
the following properties\textup{:}
\begin{itemize}
\item[\textup{(A1)}] It is trivial on $\mathrm{Cl}(K_{(N)}/H_{N})$.
\item[\textup{(A2)}] $\chi(C)\neq1$ for any chosen
$C\in\mathrm{Cl}(K_{(N)}/H)\setminus\mathrm{Cl}(K_{(N)}/H_{N})$.
\item[\textup{(A3)}] Every prime ideal factor of $(N)$
divides the conductor $(N)_\chi$.
\end{itemize}
\end{lemma}
\begin{proof}
See \cite[Lemma 3.4 and Remark 4.5]{K-S-Y2}.
\end{proof}

\begin{lemma}\label{nontrivial}
Suppose that $F$ is properly contained in $K_{(N)}$.
Then, there is a character $\rho$ of $\mathrm{Cl}(N)$
satisfying the following properties\textup{:}
\begin{itemize}
\item It is trivial on $\mathrm{Cl}(K_{(N)}/H)$, and so
$(N)_\rho=\mathcal{O}_K$.
\item It is nontrivial on $\mathrm{Cl}(K_{(N)}/F)$.
\end{itemize}
Here, $(N)_\rho$ stands for the conductor of the character $\rho$.
\end{lemma}
\begin{proof}
Since $|\mathrm{Cl}(K_{(N)}/F)|\geq2$ and
$\mathrm{Cl}(K_{(N)}/H)\cap\mathrm{Cl}(K_{(N)}/F)=\{C_1\}$ by (\ref{intersection}), one can
take a class $C\in\mathrm{Cl}(K_{(N)}/F)\setminus\mathrm{Cl}(K_{(N)}/H)$.
Thus, if we let $\mu:\mathrm{Cl}(N)
\rightarrow\mathrm{Cl}(N)/\mathrm{Cl}(K_{(N)}/H)$
be the canonical homomorphism, then there is a character $\psi$ of $\mathrm{Cl}(N)/\mathrm{Cl}(K_{(N)}/H)$
such that $\psi(\mu(C))\neq1$.
\par
Now, defining a character $\rho$ of $\mathrm{Cl}(N)$ by
$\rho=\psi\circ\mu$, we see that it is trivial on $\mathrm{Cl}(K_{(N)}/H)$.
Since
\begin{equation*}
\mathrm{Cl}(N)/\mathrm{Cl}(K_{(N)}/H)
\simeq\mathrm{Cl}(H/K)=\mathrm{Cl}(\mathcal{O}_K),
\end{equation*}
we get
$(N)_\rho=\mathcal{O}_K$.
Moreover, $\rho(C)=\psi(\mu(C))\neq1$ implies that
$\rho$ is nontrivial on $\mathrm{Cl}(K_{(N)}/F)$.
\end{proof}

\begin{proposition}\label{8972}
Assume that
\begin{equation}\label{assumption}
\gcd(72,\,N)\in\{1,\,8,\,9,\,72\}~\textrm{and $F$ is properly contained in $K_{(N)}$}.
\end{equation}
Then, there
is a character $\chi$ of $\mathrm{Cl}(N)$
and an integer $t$
which satisfy the following properties\textup{:}
\begin{enumerate}
\item[\textup{(B1)}] $\chi$ is nontrivial on $\mathrm{Cl}(K_{(N)}/F)$.
\item[\textup{(B2)}] $\gcd(N,\,t)=1$ and $t\not\equiv\pm1\Mod{N}$.
\item[\textup{(B3)}] $S(\overline{\chi},\,\xi_t)\neq0$.
\end{enumerate}
\end{proposition}
\begin{proof}
We divide the proof into three cases in accordance with $\gcd(72,N)$.
\begin{enumerate}
\item[Case 1.] First, consider the case where $\gcd(72,\,N)\in\{8,\,72\}$.
Let $C$ be the class in $\mathrm{Cl}(N)$
containing the ideal $((N/2)\tau_K+1)$.
We observe by Lemma \ref{ringGalois} that
\begin{equation}\label{Cnotin}
C\in\mathrm{Gal}(K_{(N)}/K_{(N/2)})\setminus
\mathrm{Gal}(K_{(N)}/H_N).
\end{equation}
Then, by Lemma \ref{character} there is a character $\chi$ of $\mathrm{Cl}(N)$
satisfying (A1)--(A3). If $\chi$ is trivial on $\mathrm{Cl}(K_{(N)}/F)$, then we replace $\chi$ by $\chi\rho$, where $\rho$ is a character of $\mathrm{Cl}(N)$ given in Lemma \ref{nontrivial}. 
The new character $\chi$ is nontrivial on $\mathrm{Cl}(K_{(N)}/F)$ and preserves the properties (A1)--(A3).
Take any integer $t$ such that
$\gcd(N,\,t)=1$ and $t\not\equiv\pm1\Mod{N}$.
Since $N$, $t+1$ and $t-1$ are all even, we see that
$N_+$ and $N_-$ divide $N/2$, from which it follows that
\begin{equation}\label{2inclusion}
\mathrm{Cl}(K_{(N)}/K_{(N/2)})
\subseteq\mathrm{Cl}(K_{(N)}/K_{(N_+)})\cap\mathrm{Cl}(K_{(N)}/K_{(N_-)}).
\end{equation}
We then achieve that
\begin{eqnarray*}
S(\overline{\chi},\,\xi_t)
&=&(N/N_+)
\sum_{\substack{B_+\in\mathrm{Cl}(N)\\\Mod{\mathrm{Cl}(K_{(N)}/K_{(N_+)})}}}
\overline{\chi}(B_+)\ln\left|g_{(N_+)}(C_{N_+,\,n_+})^{\sigma(B_+)}\right|
\sum_{A_+\in\mathrm{Cl}(K_{(N)}/K_{(N_+)})}\overline{\chi}(A_+)\\
&&+(N/N_-)\sum_{\substack{B_-\in\mathrm{Cl}(N)\\\Mod{\mathrm{Cl}(K_{(N)}/K_{(N_-)})}}}
\overline{\chi}(B_-)\ln\left|g_{(N_-)}(C_{N_-,\,n_-})^{\sigma(B_-)}\right|
\sum_{A_-\in\mathrm{Cl}(K_{(N)}/K_{(N_-)})}\overline{\chi}(A_-)\\
&&-2(\chi(C_t)+1)S(\overline{\chi})\quad\textrm{by Lemma \ref{summation}}\\
&=&-2(\chi(C_t)+1)S(\overline{\chi})
\quad\textrm{since $\chi$ is nontrivial on
$\mathrm{Cl}(K_{(N)}/K_{(N_+)})$ and
$\mathrm{Cl}(K_{(N)}/K_{(N_-)})$}\\
&&\hspace{4.0cm}\textrm{by (\ref{Cnotin}), (\ref{2inclusion}) and (A2)}\\
&=&-4S(\overline{\chi})\quad\textrm{by (A1) and Lemma  \ref{ringGalois}}\\
&\neq&0\quad\textrm{by Proposition \ref{Kronecker} and Remark \ref{Eulerfactor}}.
\end{eqnarray*}
\item[Case 2.] Second, consider the case where $\gcd(72,\,N)=9$. If we let
$C$ be the class in $\mathrm{Cl}(N)$ containing
the ideal $((N/3)\tau_K+1)$,
then we see that
\begin{equation}\label{Cnotin2}
C\in\mathrm{Gal}(K_{(N)}/K_{(N/3)})\setminus
\mathrm{Gal}(K_{(N)}/H_N)
\end{equation}
by Lemma \ref{ringGalois}.
By Lemma \ref{character}, there exists a character $\chi$ of $\mathrm{Cl}(N)$
satisfying (A1)--(A3).
In a similar way to the above Case 1, we
may assume that $\chi$ is nontrivial on $\mathrm{Cl}(K_{(N)}/F)$.
Take $t=2$, and then we get
\begin{equation*}
n_+=1,~N_+=\frac{N}{3}\quad
\textrm{and}\quad
n_-=1,~N_-=N.
\end{equation*}
So, we derive that
\begin{eqnarray*}
S(\overline{\chi},\,\xi_t)
&=&3\sum_{\substack{B_+\in\mathrm{Cl}(N)\\\Mod{\mathrm{Cl}(K_{(N)}/K_{(N/3)})}}}
\overline{\chi}(B_+)\ln\left|g_{(N/3)}(C_{(N/3),\,1})^{\sigma(B_+)}\right|
\sum_{A_+\in\mathrm{Cl}(K_{(N)}/K_{(N/3)})}\overline{\chi}(A_+)\\
&&+S(\overline{\chi})-2(\chi(C_t)+1)S(\overline{\chi})\quad\textrm{by Lemma \ref{summation}}\\
&=&-(2\chi(C_t)+1)S(\overline{\chi})
\quad\textrm{since $\chi$ is nontrivial
on $\mathrm{Cl}(K_{(N)}/K_{(N/3)})$
by (\ref{Cnotin2}) and (A2)}\\
&=&-3S(\overline{\chi})\quad\textrm{by (A1) and Lemma  \ref{ringGalois}}\\
&\neq&0\quad\textrm{by Proposition \ref{Kronecker} and Remark \ref{Eulerfactor}}.
\end{eqnarray*}
\item[Case 3.] Lastly, consider the case where $\gcd(72,\,N)=1$.
By Lemma \ref{character}, there is a character $\chi$ of $\mathrm{Cl}(N)$
satisfying (A1)--(A3) for any chosen
$C\in\mathrm{Cl}(K_{(N)}/H)\setminus\mathrm{Cl}(K_{(N)}/H_N)$.
In like manner as above, we
may assume that $\chi$ is nontrivial on $\mathrm{Cl}(K_{(N)}/F)$.
Take $t=2$, then it follows that
\begin{equation*}
n_+=3,~N_+=N\quad
\textrm{and}\quad
n_-=1,~N_-=N.
\end{equation*}
Therefore, we obtain
\begin{eqnarray*}
S(\overline{\chi},\,\xi_t)
&=&\chi(C_{n_+})S(\overline{\chi})
+S(\overline{\chi})-2(\chi(C_t)+1)S(\overline{\chi})\quad\textrm{by Lemma \ref{summation}}\\
&=&-2S(\overline{\chi})\quad\textrm{by (A1) and Lemma  \ref{ringGalois}}\\
&\neq&0\quad\textrm{by Proposition \ref{Kronecker} and Remark \ref{Eulerfactor}}.
\end{eqnarray*}
This proves the lemma.
\end{enumerate}
\end{proof}

\begin{lemma}\label{tlemma}
Assume that 
\begin{equation}\label{234612182436}
\gcd(72,\,N)\in\{2,\,3,\,4,\,6,\,12,\,18,\,24,\,36\}\quad\textrm{and}\quad
N\neq2,\,3,\,4,\,6.
\end{equation}
Then, there exists an integer $t$ satisfying the following properties\textup{:}
\begin{enumerate}
\item[\textup{(C1)}] $\gcd(N,\,t)=1$ and $t\not\equiv\pm1\Mod{N}$.
\item[\textup{(C2)}] There are prime factors $p_+$, $p_-$ of $N$ \textup{(}not necessarily distinct\textup{)}
such that $\gcd(p_{\pm},\,N_\pm)=1$ \textup{(}Note that $N_\pm$ depends on the choice of $t$\textup{)}.
\end{enumerate}
\end{lemma}
\begin{proof}
Let $\ell$ be an integer such that $\ell>1$ and $\gcd(6,\,\ell)=1$. 
One can take $t$ 
as listed in Table \ref{table1}.

\begin{table}[h]
\begin{center}
\begin{tabular}{c|c|c|c|c|c}\hline
$N$ & $t$ & $N_+$ & $N_-$ & $p_+$ & $p_-$ \\\hline\hline
$12$ & $5$ & $2$ & $3$ & $3$ & $2$ \\\hline
$18$ & $5$ & $3$ & $9$ & $2$ & $2$ \\\hline
$24$ & $7$ & $3$ & $4$ & $2$ & $3$ \\\hline
$36$ & $17$ & $2$ & $9$ & $3$ & $2$ \\\hline
$2\ell$ & $\ell+2$ & $\ell$ & $\ell$ & $2$ & $2$ \\\hline
$4\ell$ & $2\ell+1$ & $\ell$ & $2$ & $2$ & a prime factor of $\ell$ \\\hline
$\begin{array}{c}
2^a3^b\ell~\textrm{with}\\
a\geq0,~b\geq1
\end{array}$ & 
$\begin{array}{c}
\textrm{a solution of}\\
\left\{\begin{array}{l}
x\equiv1\Mod{2^a\ell},\\
x\equiv-1\Mod{3^b}
\end{array}\right.
\end{array}$ & 
a divisor of $2^a\ell$ & a divisor of $3^b$ & $3$ &
a prime factor of $\ell$\\\hline
 \end{tabular}
\caption{An integer $t$ satisfying (C1) and (C2)}\label{table1}
\end{center}
\end{table}
\end{proof}

Let
$(N)=\prod_{\mathfrak{p}}\mathfrak{p}^{n_\mathfrak{p}}$ be
the prime ideal factorization of $(N)$.
Then we get
\begin{equation*}
[K_{(N)}:H]=\frac{\omega(N)}{2}\prod_{\mathfrak{p}\,|\,(N)}
\left(\mathrm{N}_{K/\mathbb{Q}}(\mathfrak{p})-1\right)
\mathrm{N}_{K/\mathbb{Q}}(\mathfrak{p})^{n_\mathfrak{p}-1},
\end{equation*}
where $\omega(N)$ is the number of roots of unity in $K$
which are congruent to $1$ modulo $(N)$ (\cite[Theorem 1 in Chapter VI]{Lang94}).
One can then readily deduce that
\begin{eqnarray*}
K_{(N)}=K_{(M)}~
\textrm{for a proper divisor $M$ of $N$}\quad
\Longleftrightarrow~
2\,\|\,N~\textrm{and $2$ splits in $K$}.
\end{eqnarray*}
In this case, we have
\begin{equation}\label{reduce}
K_{(N)}=K_{(N/2)}.
\end{equation}
Furthermore, it is well known that
\begin{equation}\label{ringdegree}
[H_N:H]=
N\prod_{p\,|\,N}\left(1-\left(\frac{d_K}{p}\right)\frac{1}{p}\right),
\end{equation}
where $(d_K/p)$ is the Legendre symbol for an odd prime $p$,
and $(d_K/2)$ is the Kronecker symbol (\cite[Theorem 7.24]{Cox}).

\begin{lemma}\label{primepower}
Assume that if $2\,\|\,N$, then $2$ does not split in $K$.
Let $p$ be a prime factor of $N$ with $p^e\,\|\,N$.
Then, there is a nontrivial character $\chi_p$ of $\mathrm{Cl}(N)$ satisfying the following properties\textup{:}
\begin{itemize}
\item It is trivial on $\mathrm{Cl}(K_{(N)}/H_{{p^e}})$,
and so $(N)_{\chi_p}$ divides $(p^e)$.
\item $(N)_{\chi_p}$ is divisible by
every prime ideal factor of $(p)$.
\end{itemize}
\end{lemma}
\begin{proof}
Note that the assumption implies $[H_{{p^e}}:H]\geq2$ by \textup{(\ref{ringdegree})}.
Therefore, the lemma is an immediate consequence of \cite[Lemma 3.3]{K-S-Y2}.
\end{proof}

\begin{proposition}\label{otherthan8972}
Assume that
\begin{equation*}
\textrm{$N$ satisfies \textup{(\ref{234612182436})}}~\textrm{and $F$ is properly contained in $K_{(N)}$}.
\end{equation*}
Under this assumption instead of \textup{(\ref{assumption})}, \textup{Proposition \ref{8972}} also holds.
\end{proposition}
\begin{proof}
Let
\begin{equation*}
\chi=\prod_{p\,|\,N}\chi_p,
\end{equation*}
where $\chi_p$ is a character of $\mathrm{Cl}(N)$ given
in Lemma \ref{primepower} for each prime factor $p$ of $N$.
If $\chi$ is trivial on $\mathrm{Cl}(K_{(N)}/F)$, then
we replace $\chi$ by $\chi\rho$ where
$\rho$ is a character of $\mathrm{Cl}(N)$ given
in Lemma \ref{nontrivial}.
Then, $\chi$ satisfies the following properties:
\begin{enumerate}
\item[(i)] It is trivial on $\mathrm{Cl}(K_{(N)}/H_N)$.
\item[(ii)] It is nontrivial on $\mathrm{Cl}(K_{(N)}/F)$.
\item[(iii)] $(N)_\chi$ is divisible by every
prime ideal factor of $(N)$.
\end{enumerate}
Now, take an integer $t$
satisfying (C1) and (C2) in Lemma \ref{tlemma}.
We then derive that
\begin{eqnarray*}
S(\overline{\chi},\,\xi_t)
&=&
(N/N_+)\sum_{\substack{B_+\in\mathrm{Cl}(N)\\\Mod{\mathrm{Cl}(K_{(N)}/K_{(N_+)})}}}
\overline{\chi}(B_+)\ln\left|g_{(N_+)}(C_{N_+,\,n_+})^{\sigma(B_+)}\right|
\sum_{A_+\in\mathrm{Cl}(K_{(N)}/K_{(N_+)})}\overline{\chi}(A_+)\\
&&+(N/N_-)\sum_{\substack{B_-\in\mathrm{Cl}(N)\\\Mod{\mathrm{Cl}(K_{(N)}/K_{(N_-)})}}}
\overline{\chi}(B_-)\ln\left|g_{(N_-)}(C_{N_-,\,n_-})^{\sigma(B_-)}\right|
\sum_{A_-\in\mathrm{Cl}(K_{(N)}/K_{(N_-)})}\overline{\chi}(A_-)\\
&&-2(\chi(C_t)+1)S(\overline{\chi})\quad\textrm{by Lemma \ref{summation}}\\
&=&-4S(\overline{\chi})\quad\textrm{because $\chi$ is nontrivial
on $\mathrm{Cl}(K_{(N)}/K_{(N_+)})$
and $\mathrm{Cl}(K_{(N)}/K_{(N_-)})$}\\
&&\hspace{2.1cm}\textrm{by (iii) and (C2) and $\chi(C_t)=1$ by (i) and Lemma \ref{ringGalois}}\\
&\neq&0\quad\textrm{by Proposition \ref{Kronecker}, Remark \ref{Eulerfactor} and (iii)}.
\end{eqnarray*}
\end{proof}

\section {Main theorem}

Now, we are ready to prove our main theorem.

\begin{theorem}\label{main}
Let $K$ be an imaginary quadratic field other than  
$\mathbb{Q}(\sqrt{-1})$ and $\mathbb{Q}(\sqrt{-3})$, and let $N>1$ be an integer  such that $N\neq2,\,3,\,4,\,6$.
Then we have
\begin{equation*}
K_{(N)}=\left\{\begin{array}{ll}
K(h(1/N)) & \textrm{if $2\nparallel N$ or $2$ does not split in $K$},\\
K(h(2/N)) & \textrm{otherwise}.
\end{array}\right.
\end{equation*}
\end{theorem}
\begin{proof}
First, consider the case where $2\nparallel N$ or $2$ does not split in $K$. 
Suppose on the contrary that $F=K\left(h(1/N)\right)$ is properly contained in $K_{(N)}$.
Then,
by Propositions \ref{8972} and \ref{otherthan8972},
there exist a character $\chi$ of $\mathrm{Cl}(N)$
and an integer $t$ such that
\begin{enumerate}
\item[(B1)] $\chi$ is nontrivial on $\mathrm{Cl}(K_{(N)}/F)$,
\item[(B2)] $\gcd(N,\,t)=1$ and $t\not\equiv\pm1\Mod{N}$,
\item[(B3)] $S(\overline{\chi},\,\xi_t)\neq0$.
\end{enumerate}
On the other hand, since $F$ is a Galois extension
of $K$, it contains the Galois conjugate $h(1/N)^{\sigma(C_t)}$
of $h(1/N)$.
We then see by Proposition \ref{transformation}
and Lemma \ref{tinvariant} that
\begin{equation*}
h(1/N)^{\sigma(C_t)}
=f_{\left[\begin{smallmatrix}0\\1/N\end{smallmatrix}\right]}(\tau_K)^{\sigma(C_t)}
=f_{(N)}(C_1)^{\sigma(C_t)}
=f_{(N)}(C_t)
=f_{\left[\begin{smallmatrix}0\\t/N\end{smallmatrix}\right]}(\tau_K)
=h(t/N).
\end{equation*}
Thus $F$ contains the element $\xi_t=\left(h(t/N)-h(1/N)\right)^{12N}$.
Now, we derive that
\begin{eqnarray*}
S(\overline{\chi},\,\xi_t)&=&
\sum_{C\in\mathrm{Cl}(N)}\overline{\chi}(C)\ln\left|\xi_t^{\sigma(C)}\right|\\
&=&\sum_{\substack{B\in\mathrm{Cl}(N)\\\Mod{
\mathrm{Cl}(K_{(N)}/F)}}}\sum_{A\in\mathrm{Cl}(K_{(N)}/F)}
\overline{\chi}(AB)\ln\left|\xi_t^{\sigma(AB)}\right|\\
&=&\sum_{\substack{B\in\mathrm{Cl}(N)\\\Mod{\mathrm{Cl}(K_{(N)}/F)}}}
\overline{\chi}(B)\ln\left|\xi_t^{\sigma(B)}\right|
\sum_{A\in\mathrm{Cl}(K_{(N)}/F)}\overline{\chi}(A)
\quad\textrm{because $\xi_t\in F$}\\
&=&0\quad\textrm{by (B1)},
\end{eqnarray*}
which contradicts (B3).
Hence, we have $K_{(N)}=K\left(h(1/N)\right)$ as desired.
\par
Second, consider the case where
$2\,\|\,N$ and $2$ splits in $K$. Then we have
\begin{eqnarray*}
K_{(N)}&=&K_{(N/2)}\quad\textrm{as mentioned in (\ref{reduce})}\\
&=&K(h(2/N))\quad\textrm{by the first case of the theorem}.
\end{eqnarray*}
This completes the proof. 
\end{proof}

\bibliographystyle{amsplain}

\begin{thebibliography}{99}

\bibitem {Cox} D. A. Cox, \textit{Primes of the form $x^2+ny^2$: Fermat, Class Field, and Complex Multiplication},
John Wiley \& Sons, Inc., New York, 1989.

\bibitem {E-K-S} I. S. Eum, J. K. Koo and D. H. Shin,
\textit{Ring class invariants over imaginary quadratic fields}, Forum Math. 28 (2016), no. 2,
201-–217.

\bibitem {F-L-R} G. Frei, F. Lemmermeyer and P. J. Roquette, \textit{Emil Artin and Helmut Hasse-the Correspondence 1923--1958}, Contributions in Mathematical and Computational Sciences, 5. Springer, Heidelberg, 2014.



\bibitem {Hasse} H. Hasse, \textit{Neue Begr\"{u}ndung der Komplexen Multiplikation I, II},
J. reine angew Math. 157, 165 (1927, 1931), 115--139, 64--88.

\bibitem {Janusz} G. J. Janusz, \textit{Algebraic Number Fields},
2nd edn, Grad. Studies in Math. 7, Amer. Math. Soc., Providence,
R. I., 1996.

\bibitem {J-K-S} H. Y. Jung, J. K. Koo and D. H. Shin,
\textit{Generation of ray class fields modulo $2$, $3$, $4$ or $6$ by
using the Weber function}, to appear in J. Korean Math. Soc.

\bibitem {K-L} D. Kubert and S. Lang, \textit{Modular Units}, Grundlehren der mathematischen Wissenschaften 244, Spinger-Verlag, New York-Berlin, 1981.

\bibitem {K-S-Y} J. K. Koo, D. H. Shin and D. S. Yoon, \textit{Generation of class fields by using the Weber function}, J. Number Theory 167 (2016), 74-–87.

\bibitem {K-S-Y2} J. K. Koo, D. H. Shin and D. S. Yoon,
\textit{Generation of ring class fields by eta-quotients}, to appear in J. Korean Math. Soc., http://arxiv.org/abs/1404.3282.

\bibitem {Lang87} S. Lang, \textit{Elliptic Functions}, With an appendix by J. Tate, 2nd edn, Grad. Texts in Math. 112, Spinger-Verlag, New York, 1987.

\bibitem {Lang94} S. Lang, \textit{Algebraic Number Theory}, 2nd edn, Spinger-Verlag, New York, 1994.

\bibitem {Ramachandra}
K. Ramachandra, \textit{Some applications of Kronecker's limit
formula}, Ann. of Math. (2) 80 (1964), 104--148.



\bibitem {Siegel} C. L. Siegel, \textit{Advanced Analytic Number Theory}, 2nd edn,
Tata Institute of Fundamental Research Studies in Mathematics 9, Tata Institute of Fundamental Research, Bombay, 1980.

\bibitem {Sugawara33} M. Sugawara, \textit{On the so-called Kronecker's dream in young days}, Proc. Phys.-Math. Japan (3) 15 (1933), 99--107.

\bibitem {Sugawara36} M. Sugawara, \textit{Zur theorie der Komplexen Multiplikation. I, II},
J. reine angew Math. 174, 175 (1936), 189--191, 65--68.




\end{thebibliography}

\address{
Department of Mathematical Sciences \\
KAIST \\
Daejeon 34141\\
Republic of Korea} {jkkoo@math.kaist.ac.kr}
\address{
Department of Mathematics\\
Hankuk University of Foreign Studies\\
Yongin-si, Gyeonggi-do 17035\\
Republic of Korea} {dhshin@hufs.ac.kr}
\address{
Department of Mathematical Sciences \\
KAIST \\
Daejeon 34141\\
Republic of Korea} {math\_dsyoon@kaist.ac.kr}

\end{document}